\documentclass[10pt,oneside,leqno]{amsart}
\usepackage{amsxtra}
\usepackage{amsopn}
\usepackage{color}
\usepackage{amsmath,amsthm,amssymb}
\usepackage{amscd}
\usepackage{amsfonts}
\usepackage{latexsym}
\usepackage{verbatim}
\usepackage{multirow}
\usepackage[all]{xy}

\theoremstyle{plain}
\newtheorem{theorem}{Theorem}[section]

\newtheorem{lemma}[theorem]{Lemma}
\newtheorem{proposition}[theorem]{Proposition}
\newtheorem{corollary}[theorem]{Corollary}
\newtheorem{remark}[theorem]{Remark}

\newtheorem{remark-question}[section]{Remark-Question}

\newcommand\fra{{\mathfrak a}} 
\newcommand\frg{{\mathfrak g}}

\newcommand\nilm{\Gamma\backslash G}

\begin{document}
\title[]{
On the real homotopy type of generalized complex nilmanifolds
}

\keywords{Nilmanifold, nilpotent Lie algebra, complex structure, symplectic form,
generalized complex structure, homotopy theory, minimal model.}
\subjclass[2000]{Primary 53D18, 55P62; Secondary 17B30, 53C56, 53D05.}
%
%

\author{Adela Latorre}
\address[A. Latorre]{Departamento de Matem\'atica Aplicada,
Universidad Polit\'ecnica de Madrid,
C/ Jos\'e Antonio Novais 10,
28040 Madrid, Spain}
\email{adela.latorre@upm.es}

\author{Luis Ugarte}
\address[L. Ugarte]{Departamento de Matem\'aticas\,-\,I.U.M.A.\\
Universidad de Zaragoza\\
Campus Plaza San Francisco\\
50009 Zaragoza, Spain}
\email{ugarte@unizar.es}

\author{Raquel Villacampa}
\address[R. Villacampa]{Centro Universitario de la Defensa\,-\,I.U.M.A., Academia General
Mili\-tar, Crta. de Huesca s/n. 50090 Zaragoza, Spain}
\email{raquelvg@unizar.es}


\maketitle

\begin{abstract}
We prove that for any $n\geq 4$ there are infinitely many real homotopy types of $2n$-dimensional nilmanifolds admitting
generalized complex structures of every type $k$, for $0 \leq k \leq n$. This is in deep contrast to the
$6$-dimensional case.
\end{abstract}

\maketitle


\section{Introduction}\label{intro}

\noindent
Nilmanifolds constitute a well-known class of compact manifolds providing interesting explicit examples
of geometric structures with special properties. A nilmanifold is a compact quotient $N=\nilm$ of a 
connected and
simply connected nilpotent Lie group $G$
by a lattice $\Gamma$ of maximal rank in $G$. Hence, any left-invariant geometric structure on $G$ descends to $N$.
We will refer to such structures as \emph{invariant}.
For instance, there are nilmanifolds admitting invariant complex structures, as the Iwasawa manifold,
or invariant symplectic forms, as the Kodaira-Thurston manifold, with remarkable properties \cite{Thu,TO-libro}.
However, by~\cite{Has-PAMS} a nilmanifold cannot admit any K\"ahler metric (invariant or not), unless it is a torus.
Since there are also nilmanifolds with no invariant complex structures or symplectic forms,
it is an interesting problem to understand which nilmanifolds do admit such kinds of structures.

Symplectic and complex geometries constitute two special cases in the unified framework given by generalized complex geometry, introduced by Hitchin in \cite{Hitchin} and further developed by Gualtieri~\cite{Gualtieri}. 
In~\cite{CG-generalized}, Cavalcanti and Gualtieri study invariant generalized complex structures on nilmanifolds. 
Furthermore, Angella, Calamai and Kasuya show in \cite{ACK} that nilmanifolds provide a nice class for investigating 
cohomological aspects of generalized complex structures.

In \cite[Theorem 3.1]{CG-generalized} the authors prove 
that any invariant generalized complex structure on a $2n$-dimensional nilmanifold must be
\emph{generalized Calabi-Yau}, extending
a result of Salamon \cite{Salamon} for invariant complex structures. That means that any generalized complex structure is given by a (left-invariant) trivialization $\rho$ of the canonical bundle, i.e.
\begin{equation}\label{gcs-1}
\rho = e^{B+i\,\omega} \, \Omega,
\end{equation}
where $B,\omega$ are real invariant 2-forms and $\Omega$ is a globally decomposable complex
$k$-form, i.e. $\Omega= \theta^1\wedge\cdots\wedge\theta^k$,
with each $\theta^i$ invariant. Moreover, these data satisfy the non-degeneracy condition
\begin{equation}\label{gcs-2}
\omega^{n-k}\wedge \Omega\wedge\overline{\Omega} \not= 0,
\end{equation}
as well as the integrability condition
\begin{equation}\label{gcs-3}
d \rho =0.
\end{equation}
The integer $0 \leq k \leq n$ is called the \emph{type} of the generalized complex structure.
Type $k=n$ corresponds to the usual complex structures, whereas
the structures of type $k=0$ are the symplectic ones.
It is proved in \cite{CG-generalized}
that all the $6$-dimensional nilmanifolds admit a generalized complex structure of type $k$, for at least one $0 \leq k \leq 3$; however, it is shown that 
there are nilmanifolds in eight dimensions not admitting any invariant generalized complex structure.

Let $\frg$ be the nilpotent Lie algebra underlying the nilmanifold $N=\nilm$, and let $(\bigwedge\frg^*,d)$
be the Chevalley-Eilenberg complex seen as a commutative differential graded algebra (CDGA).
Hasegawa proved in~\cite{Has-PAMS} that this CDGA provides not only the $\mathbb{R}$-minimal model of $N$ 
but also its $\mathbb{Q}$-minimal model. A result by 
Bazzoni and Mu\~noz
asserts that, in six dimensions, there are infinitely many \emph{rational} homotopy types of nilmanifolds,
but only 34 different real homotopy types (see~\cite[Theorem~2]{BM2012}).
Hence, 
there only exits a finite number of real homotopy types of $6$-dimensional nilmanifolds admitting a generalized complex structure.  
In this paper we prove that the situation is completely different in eight dimensions:

\begin{theorem}\label{infinite-GCS}
There are infinitely many real homotopy types of $8$-dimensional nilmanifolds admitting
generalized complex structures of every type $k$, for $0 \leq k \leq 4$.
\end{theorem}

It is worth remarking that the existence of infinitely many real homotopy types of $8$-dimensional nilmanifolds with complex structure (having special Hermitian metrics) is proved in~\cite{LUV-ComplexManifolds}. However, such nilmanifolds do not admit any symplectic form. 

\vskip.3cm

This paper is structured as follows. In Section~\ref{nilvariedades} we review some general results about minimal models and homotopy theory, and we define a family of nilmanifolds $N_{\alpha}$ in eight dimensions depending on a rational parameter $\alpha>0$. Section~\ref{GCS} is devoted to the construction of  generalized complex structures on the nilmanifolds $N_{\alpha}$.
More concretely, we prove the following:

\begin{proposition}\label{prop-1}
For each $\alpha\in {\mathbb{Q}}^+$, the $8$-dimensional nilmanifold $N_{\alpha}$ has generalized complex structures of every type $k$, for $0 \leq k \leq 4$.
\end{proposition}

In Section~\ref{no-iso} we study the real homotopy types of the nilmanifolds in the family $\{N_{\alpha}\}_{\alpha\in {\mathbb{Q}}^+}$. More precisely, the result below is attained:

\begin{proposition}\label{prop-2}
If $\alpha \not= \alpha'$, then the nilmanifolds $N_{\alpha}$ and $N_{\alpha'}$ have non-isomorphic $\mathbb{R}$-minimal models, so they have different real homotopy type.
\end{proposition}

Note that Theorem~\ref{infinite-GCS} is a direct consequence of Propositions~\ref{prop-1} and~\ref{prop-2}. 
Moreover, by taking products with even dimensional tori the result holds in any dimension $2n\geq8$ and for every $0 \leq k \leq n$.
Furthermore, our result in Theorem~\ref{infinite-GCS} can be extended to the 
\emph{complex homotopy} setting (see Remark~\ref{C-minimal-model} 
for details). Since the nilmanifolds $N_{\alpha}$ 
cannot admit any K\"ahler metric~\cite{Has-PAMS}, one has the following

\begin{corollary}\label{cor-infinite-GCS}
Let $n\geq 4$. There are infinitely many complex homotopy types of $2n$-dimensional compact non-K\"ahler manifolds admitting
generalized complex structures 
of every type $k$, for $0 \leq k \leq n$.
\end{corollary}

\section{The family of nilmanifolds $N_{\alpha}$}\label{nilvariedades}

\noindent 
Let us start recalling some general results about homotopy theory and minimal models, with special attention to the class of nilmanifolds.
In \cite{Sullivan} Sullivan shows  that it is possible to associate a
minimal model to any nilpotent CW-complex $X$, i.e. a space $X$ whose
fundamental group $\pi_1(X)$ is a nilpotent group that acts in a nilpotent way on the
higher homotopy group $\pi_k(X)$ of $X$ for every $k > 1$. 
Recall that a minimal model is a commutative differential graded algebra, CDGA for short, $(\bigwedge V_X,d)$ defined over the rational numbers $\mathbb{Q}$ and
satisfying a certain minimality condition,
that encodes the rational homotopy type of~$X$~\cite{GM-libro}.

More generally, 
let $\mathbb{K}$ be the field $\mathbb{Q}$ or $\mathbb{R}$. 
A CDGA $(\bigwedge V,d)$ defined over
$\mathbb{K}$ is said to be minimal if the following conditions hold:

\vskip.2cm

\noindent\ \emph{(i)} $\bigwedge V$ is the free commutative algebra generated by the graded vector
space $V = \oplus V^l$;

\vskip.2cm

\noindent\ \emph{(ii)} there exists a basis $\{x_j\}_{j \in J}$, for some well-ordered index set $J$, such
that $\deg(x_k) \leq \deg(x_j)$ 

\quad for $k < j$, and each $dx_j$ is expressed in terms of the
preceding $x_k$ $(k < j)$.

\vskip.3cm

A $\mathbb{K}$-minimal model of a differentiable manifold $M$ is a minimal CDGA $(\bigwedge V,d)$ over~$\mathbb{K}$
together with a quasi-isomorphism $\phi$ from $(\bigwedge V,d)$ to the $\mathbb{K}$-de Rham complex of $M$,
i.e. a morphism~$\phi$ inducing an isomorphism in cohomology.
Here, the $\mathbb{K}$-de Rham complex of $M$ is the usual de Rham complex of differential forms $(\Omega^*(M),d)$ when $\mathbb{K}=\mathbb{R}$,
whereas for $\mathbb{K}=\mathbb{Q}$ one considers $\mathbb{Q}$-polynomial forms instead. 
Notice that the $\mathbb{K}$-minimal model is unique up to isomorphism, since ${\rm char}\, (\mathbb{K})=0$. 
By \cite{DGMS,Sullivan}, two nilpotent manifolds $M_1$ and $M_2$ have the same $\mathbb{K}$-homotopy type
if and only if their $\mathbb{K}$-minimal models are isomorphic.
It is clear that if 
$M_1$ and $M_2$ have different real homotopy type, then $M_1$ and $M_2$ also have different rational homotopy type.

Let $N$ be a nilmanifold,
i.e. $N=\nilm$ is a compact quotient of a connected and simply connected
nilpotent Lie group $G$ by a lattice $\Gamma$ of maximal rank.
For any nilmanifold $N$ one has $\pi_1(N)=\Gamma$, 
which is nilpotent, and $\pi_k(N)=0$ for every $k \geq 2$. Therefore, nilmanifolds are nilpotent spaces.

Let $n$ be the dimension of the nilmanifold $N=\nilm$, and let $\frg$ be the Lie algebra of $G$. It is well known that
the minimal model of $N$ is given by the Chevalley-Eilenberg complex $(\bigwedge\frg^*,d)$ of $\frg$. 
Recall that by \cite{Malcev}
the existence of a lattice $\Gamma$ of maximal rank in $G$ is equivalent to
the nilpotent Lie algebra~$\frg$ being rational, i.e. there exists a basis $\{e^1,\ldots,e^n\}$ for the dual $\frg^*$ such that
the structure constants are rational numbers.
Thus, the rational and the nilpotency conditions of the Lie algebra~$\frg$
allow to take
a basis $\{e^1,\ldots,e^n\}$ for $\frg^*$ satisfying
\begin{equation}\label{st-ecus}
de^1=de^2=0, \quad de^j=\sum_{i,k<j} a^j_{ik}\, e^i\wedge e^k \ \ \mbox{ for } j=3,\ldots, n,
\end{equation}
with structure constants $a^j_{ik} \in \mathbb{Q}$.

Therefore, $(\bigwedge\frg^*,d)$ is a CDGA satisfying both conditions \emph{(i)} and \emph{(ii)}
with ordered index set  $J=\{1,\ldots,n\}$
and $V=V^1=\langle x_1,\ldots,x_n \rangle= \sum_{j=1}^n \mathbb{Q}x_j $, where $x_j=e^j$ for $1 \leq j \leq n$.
That is to say,  the CDGA $(\bigwedge\frg^*,d)$ over $\mathbb{Q}$ is minimal, and 
it is determined by
\begin{equation}\label{Q-minimal}
\big(\bigwedge\langle x_1,\ldots,x_n\rangle,\,d\,\big)
\end{equation}
with $n$ generators $x_1,\ldots,x_n$ of degree $1$ satisfying equations of the form (\ref{st-ecus}).
Notice that the CDGA $(\bigwedge\frg^*,d)$ over $\mathbb{R}$ is also minimal, since
it is given by
\begin{equation}\label{R-minimal}
\big(\bigwedge\langle x_1,\ldots,x_n\rangle \otimes \mathbb{R},\,d\,\big).
\end{equation}

There is a canonical morphism $\phi$ from the Chevalley-Eilenberg complex $(\bigwedge\frg^*,d)$
to the de Rham complex $(\Omega^*(\nilm),d)$ of the nilmanifold. Nomizu proves in \cite{Nomizu} that $\phi$ induces an isomorphism in cohomology, so the $\mathbb{R}$-minimal model of the nilmanifold
$N=\nilm$ is given by \eqref{R-minimal}. 
Hasegawa observes in~\cite{Has-PAMS} that \eqref{Q-minimal} is the $\mathbb{Q}$-minimal model of $N$ and that, conversely, given a $\mathbb{Q}$-minimal CDGA of the form \eqref{Q-minimal}, there exists a nilmanifold $N$ with \eqref{Q-minimal} as its $\mathbb{Q}$-minimal model. 

Deligne, Griffiths, Morgan and Sullivan prove in
\cite{DGMS} that the $\mathbb{K}$-minimal model, $\mathbb{K}=\mathbb{Q}$ or $\mathbb{R}$, of a compact K\"ahler manifold
is formal, i.e.  it is quasi-isomorphic to its cohomology. Hasegawa shows in~\cite{Has-PAMS} that the minimal model \eqref{Q-minimal} is formal 
if and only if all the structure constants $a^j_{ik}$ in~(\ref{st-ecus}) vanish, 
so a symplectic nilmanifold does not admit any K\"ahler metric unless it is a torus.
(See for instance~\cite{FHT-libro,TO-libro} for more results on homotopy theory and applications to symplectic geometry.)

\vskip.2cm

Bazzoni and Mu\~noz study in \cite{BM2012} the $\mathbb{K}$-homotopy types of nilmanifolds of low dimension. They prove that, up to dimension $5$, the number of rational homotopy types of nilmanifolds is finite.  
However, in six dimensions the following result holds:

\begin{theorem}\label{teorema-BM}\cite[Theorem 2]{BM2012}
There are infinitely many rational
homotopy types of $6$-dimensional nilmanifolds, but there are only $34$ real homotopy types of $6$-dimensional nilmanifolds.
\end{theorem}

As a direct consequence, there is only a finite number of real homotopy types of $6$-dimensional nilmanifolds
admitting an extra geometric structure of any kind. In the case of generalized complex structures (see Section~\ref{GCS} for definition), we will prove that, in contrast to the $6$-dimensional case, 
there are infinitely many real homotopy types of $8$-dimensional nilmanifolds admitting
generalized complex structures of every type $k$, for $0 \leq k \leq 4$.

To construct such nilmanifolds, let us take a positive rational number $\alpha$
and consider the connected, simply connected,
nilpotent Lie group $G_{\alpha}$ corresponding to the nilpotent Lie algebra $\frg_{\alpha}$ defined by
\begin{equation}\label{equations}
\begin{cases}
de^1=de^2 = de^3=de^4=0,\\[3pt]
de^5=e^{12},\\[3pt]
de^6=e^{15} + (1-\alpha)\, e^{24},\\[3pt]
de^7=-(1+\alpha)\,e^{14} - e^{23} + (1+\alpha)\,e^{25},\\[3pt]
de^8=e^{16}+e^{27} + e^{34} -2\,e^{45},
\end{cases}
\end{equation}
where $e^{ij}=e^i\wedge e^j$, being $\{e^i\}_{i=1}^8$ a basis for $\frg^{\ast}_{\alpha}$, and $\alpha\in {\mathbb{Q}}^+$.
It is clear from~\eqref{equations} that the Lie algebra~$\frg_{\alpha}$ is rational, hence by Mal'cev theorem \cite{Malcev},
there exists a lattice $\Gamma_{\alpha}$ of maximal rank in $G_{\alpha}$.
We denote by $N_{\alpha}=\Gamma_{\alpha}\backslash G_{\alpha}$ the
corresponding compact quotient.

Therefore, we have defined a family of nilmanifolds $N_{\alpha}$ of dimension $8$ depending on the rational parameter $\alpha\in {\mathbb{Q}}^+$. We will study the properties of $N_{\alpha}$, for $\alpha\in {\mathbb{Q}}^+$, in Sections~\ref{GCS} and~\ref{no-iso}. Here, we simply provide their Betti numbers.

A direct calculation using Nomizu's theorem \cite{Nomizu} allows to explicitly compute the de Rham cohomology groups
of any nilmanifold $N_{\alpha}$. In particular, for degrees $1\leq l\leq 3$, the $l$-th de Rham cohomology groups $H_{\text{dR}}^l(N_{\alpha})$ are 
\begin{equation*}
\begin{split}
H_{\text{dR}}^1(N_{\alpha})=\langle\, &[e^1],\,[e^2],\,[e^3],\,[e^4] \,\rangle, \\[5pt]
H_{\text{dR}}^2(N_{\alpha})=\langle\, &[e^{13}],\,[e^{14}],\,[e^{23}],\,[e^{24}],\,[e^{34}],\,
	[e^{16}-(1-\alpha)\,e^{45}],\\
	&[e^{17}+(1+\alpha)\,e^{26} + e^{35}],\,
	\big[(1+\alpha)e^{18}-e^{37}-(1+\alpha)(3+\alpha)\,e^{46} +(1+\alpha)e^{57}\big] 
	\,\rangle, \\[5pt]
H_{\text{dR}}^3(N_{\alpha})= \langle\, &[e^{127}],\,[e^{146}],\,[e^{256}],\,[e^{248}+e^{456}],\,
	[e^{128}+2\,e^{246}-e^{345}], [e^{136}+(1-\alpha)\,e^{345}], \\
&
[e^{137}+(1+\alpha)\,e^{236}],\,[e^{156}-(1-\alpha)\,e^{246}],\,
	[e^{138}+(3+\alpha)\,e^{346}-e^{357}],\\
&
[e^{147}+(1+\alpha)\,e^{246}+e^{345}],\,[(1+\alpha)e^{148}-e^{347}-(1+\alpha)e^{457}],\\
&
[e^{238}+e^{356}+(3-\alpha)\,(e^{148}-e^{457})] \,\rangle.
\end{split}
\end{equation*}
Let $b_l(N_{\alpha})$ denote the $l$-th Betti number of $N_{\alpha}$.
By 
duality we have: 
$$
b_0(N_{\alpha})=b_8(N_{\alpha})=1, \quad b_1(N_{\alpha})=b_7(N_{\alpha})=4, \quad
b_2(N_{\alpha})=b_6(N_{\alpha})=8, \quad b_3(N_{\alpha})=b_5(N_{\alpha})=12.
$$
One can finally compute the Betti number $b_4(N_{\alpha})$ taking into account that the Euler-Poincar\'e characteristic $\chi$ of
a nilmanifold always vanishes, namely,
$$
0=\chi(N_{\alpha}) = \sum_{l=0}^8 (-1)^l b_l(N_{\alpha})
= b_4(N_{\alpha}) + 2\big( b_0(N_{\alpha})-b_1(N_{\alpha})+b_2(N_{\alpha})-b_3(N_{\alpha}) \big),
$$
which implies $b_4(N_{\alpha})=14$. In particular, we observe that the Betti numbers of the nilmanifolds $N_{\alpha}$ do not depend on $\alpha$.

\section{Generalized complex structures on the nilmanifolds $N_{\alpha}$}\label{GCS}

\noindent
Generalized complex geometry, in the sense of Hitchin and Gualtieri~\cite{Hitchin,Gualtieri}, establishes a unitary framework for symplectic 
and complex geometries.
Let $M$ be a compact differentiable manifold of dimension~$2n$. 
Denote by $TM$ the tangent bundle and by $T^*M$ the cotangent bundle, and
consider the vector bundle $TM\oplus T^*M$ endowed with the natural symmetric pairing 
$$
\langle X+\xi \mid Y+\eta \rangle=\frac{1}{2} \big( \xi(Y)+\eta(X) \big).
$$
Recall that the \emph{Courant bracket} on the space ${\mathcal{C}}^{\infty}(TM\oplus T^*M)$ is given by
$$
[X+\xi , Y+\eta] = [X,Y] + {\mathcal L}_X\eta - {\mathcal L}_Y\xi -\frac{1}{2} d(\iota_X\eta - \iota_Y\xi),
$$
where ${\mathcal L}$ and $\iota$ respectively denote the Lie derivative and the interior product.
A \emph{generalized complex structure} on $M$ is an endomorphism 
${\mathcal J}\in  {\rm End}(TM\oplus T^*M)$ satisfying ${\mathcal J}^2=-1$ whose $i$-eigenbundle $L\subset (TM\oplus T^*M)\otimes {\mathbb{C}}$ is involutive with respect to the Courant bracket.

There is an action of $TM\oplus T^*M$ on $\bigwedge^{\bullet} T^*M$ 
given by
$$
(X+\xi)\cdot \rho = \iota_X\rho+\xi\wedge\rho.
$$
Now, for a generalized complex structure $\mathcal J$ with $i$-eigenbundle $L$, one can define the \emph{canonical line bundle} $K \subset \bigwedge^{\bullet} T^*M \otimes {\mathbb{C}}$ as
$$
L={\rm Ann}(K)=\{u \in (TM\oplus T^*M)\otimes {\mathbb{C}} \mid u\cdot K=0 \}.
$$
Any $\rho\in K$ is a non-degenerate pure form, i.e. 
it can be written as
$$
\rho = e^{B+i\,\omega}\, \Omega,
$$
where $B,\omega$ are real $2$-forms and $\Omega=\theta^1\wedge\cdots\wedge\theta^k$ is a complex decomposable $k$-form, such that 
$$
\omega^{n-k}\wedge\Omega\wedge\overline{\Omega} \not= 0.
$$ 
The number $k$ is called the \emph{type} of the generalized complex structure.
Moreover, any $\phi\in {\mathcal C}^{\infty}(K)$ is integrable, i.e. there exists $X+\xi \in {\mathcal{C}}^{\infty}(TM\oplus T^*M)$ satisfying 
$$
d \phi = (X+\xi) \cdot \phi.
$$
Notice that the converse also holds: if $K\subset \bigwedge^{\bullet} T^*M \otimes {\mathbb{C}}$ is a line bundle such that any $\rho\in K$ is a non-degenerate pure form and any $\phi\in {\mathcal C}^{\infty}(K)$ is integrable, then we have a generalized complex structure whose $i$-eigenbundle is $L={\rm Ann}(K)$.

In the case that $K$ is a trivial bundle admitting a nowhere vanishing closed section, the generalized complex structure is called \emph{generalized Calabi-Yau}.

Recall that if $J$ is a complex structure on $M$
then
$${\mathcal{J}}_J=\left(
\begin{array}{cc}
 -J & 0  \\
 0 & J^*  \\
\end{array}
\right)$$
is a generalized complex structure of type $n$, and if $\omega$ is a symplectic form on $M$ then
$${\mathcal{J}}_{\omega}=\left(
\begin{array}{cc}
 0 & -\omega^{-1}  \\
 \omega & 0  \\
\end{array}
\right)$$
is a generalized complex structure of type $0$.
Near a regular point (i.e. a point where the type is locally constant) a generalized complex structure is equivalent to a product of a complex and a symplectic structure 
\cite[Theorem 3.6]{Gualtieri}.

\vskip.12cm

In the case of a nilmanifold $N$, Cavalcanti and Gualtieri proved in \cite[Theorem 3.1]{CG-generalized} that any \emph{invariant} generalized complex structure on $N$ must be generalized Calabi-Yau. Hence, it is given by a (left-invariant) trivialization $\rho$ of the canonical bundle of the form (\ref{gcs-1}) satisfying the 
non-degeneracy condition~(\ref{gcs-2})
and the integrability condition (\ref{gcs-3}).

\vskip.2cm

Let us now prove Proposition~\ref{prop-1}, that is, 
each nilmanifold $N_{\alpha}$ has 
generalized complex structures of every type $k$, for $0 \leq k \leq 4$.
These structures will be explicitly described in terms of the global basis of invariant 1-forms 
$\{ e^i \}_{i=1}^8$ on $N_{\alpha}$ given in~\eqref{equations}. We begin providing a structure of type 4.

\vskip.4cm\noindent
$\bullet$ \textbf{Generalized complex structure of type 4 (complex structure).}
We define the following complex $1$-forms:
\begin{equation}\label{complex}
\begin{array}{cl}
\theta^1 \!\!\! &= \displaystyle{\frac 12 \left( \frac{1}{\sqrt{3+\alpha}}\,e^1 - e^2 \right)
		+ \frac{i}{2} \left( \frac{1}{\sqrt{3+\alpha}}\,e^1 + e^2 \right)} ,\\[10pt]
\theta^2 \!\!\! &= \displaystyle{-\alpha\,e^4 - \frac{i}{\sqrt{3+\alpha}}\left( \frac 12\,e^3 + e^5 \right) },\\[8pt]
\theta^3 \!\!\! &= \displaystyle{\frac{\alpha}{(1+\alpha)^2} \left( e^6 - \frac{1}{\sqrt{3+\alpha}}\,e^7 \right)
		+ \frac{i\,\alpha}{(1+\alpha)^2} \left( e^6 + \frac{1}{\sqrt{3+\alpha}}\,e^7 \right)} ,\\[10pt]
\theta^4 \!\!\! &= \displaystyle{\frac{1}{\sqrt{3+\alpha}} \left( e^5 + \frac{1-\alpha}{2\,(1+\alpha)^2}\, e^3 \right)
		- i\,\alpha \left( e^4 + \frac{2}{(1+\alpha)^2\sqrt{3+\alpha}}\,e^8 \right)} .
\end{array}
\end{equation}

From the equations (\ref{equations}) we get
\begin{equation}\label{FI-equations}
\left\{\begin{array}{ccl}
d\theta^1 \!\!&\!\!=\!\!&\!\! 0,\\[6pt]
d\theta^2 \!\!&\!\!=\!\!&\!\! \theta^{1}\wedge \overline{\theta^{1}},\\[5pt]
d\theta^3 \!\!&\!\!=\!\!&\!\! \displaystyle{\theta^{1}\wedge \theta^{4} + \theta^{1}\wedge \overline{\theta^{4}} + \frac{2\,\alpha}{(1+\alpha)^2}\,\theta^{2}\wedge \overline{\theta^{1}}
+ \frac{2\,i}{(1+\alpha)^2}\,\theta^{1}\wedge \overline{\theta^{2}}},\\[9pt]
d\theta^4 \!\!&\!\!=\!\!&\!\! \displaystyle{i\,\theta^{1}\wedge \overline{\theta^{1}} -\frac{2}{(1+\alpha)^2}\,\theta^{2}\wedge \overline{\theta^{2}} - i\,\theta^{1}\wedge \overline{\theta^{3}}+i\,\theta^{3}\wedge \overline{\theta^{1}}}.
\end{array}\right.
\end{equation}

Declaring the forms $\theta^i$ to be of bidegree $(1,0)$, we obtain an almost complex structure $J$ on the nilmanifold $N_{\alpha}$ for every $\alpha\in {\mathbb{Q}}^+$. It follows from (\ref{FI-equations}) that $d\theta^i$ has no $(0,2)$ component, so $J$ is integrable.
Hence,
$\rho =\Omega= \theta^1\wedge\theta^2\wedge\theta^3\wedge\theta^4$ 
is a generalized complex structure of type 4.

\begin{remark}\label{hol-Poisson}
{\rm
The complex nilmanifold $(N_{\alpha},J)$ has a holomorphic Poisson structure given by the holomorphic bivector $\beta=X_3 \wedge X_4$ of rank two, where $\{X_i\}$ is the dual basis of $\{\theta^i\}$ (see \cite[Theorem~5.1]{CG-generalized} for the existence of such a bivector on nilmanifolds).
It is worth observing that $(N_{\alpha},J)$ does not admit any
(invariant or not) holomorphic symplectic structure: since the center of the Lie algebra $\frg_\alpha$ has dimension 1,
the complex structure defined by (\ref{complex}) is \emph{strongly non-nilpotent} (see \cite{LUV-SnN} for properties on this kind of complex structures); by 
\cite{BFLM}, a strongly non-nilpotent complex structure on an $8$-dimensional nilmanifold 
cannot support any holomorphic symplectic form.
}
\end{remark}

\vskip.4cm\noindent
$\bullet$ \textbf{Generalized complex structure of type 3.}
Let us consider $\rho = e^{B+i\,\omega} \, \Omega$, with
$B=0$, $\omega =i\, \theta^4\wedge\overline{\theta}^{\,4}$ and
$\Omega=\theta^1\wedge\theta^2\wedge\theta^3$, 
where $\theta^1,\theta^2,\theta^3$ and $\theta^4$
are the complex 1-forms given in \eqref{complex}.
It is clear that
$$
\omega\wedge\Omega\wedge\overline{\Omega} = -i\,\theta^1\wedge\theta^2\wedge\theta^3\wedge\theta^4\wedge \overline{\theta^{1}}\wedge \overline{\theta^{2}}\wedge \overline{\theta^{3}}\wedge \overline{\theta^{4}} \not= 0,
$$
so the non-degeneracy condition~\eqref{gcs-2} is satisfied.
A direct calculation using \eqref{FI-equations} shows  $$
d(\theta^1\wedge\theta^2\wedge\theta^3)=0
$$ 
and
$$
d\omega\wedge\theta^1\wedge\theta^2\wedge\theta^3 = i\, d\theta^4\wedge\bar\theta^4\wedge\theta^1\wedge\theta^2\wedge\theta^3 - i\,\theta^4\wedge d\, \overline{\theta^{4}}\wedge\theta^1\wedge\theta^2\wedge\theta^3 =0,
$$
which implies that $d\rho=0$, i.e. the integrability condition \eqref{gcs-3} holds.

Therefore, the nilmanifolds $N_{\alpha}$ have generalized complex structures of type 3.

\vskip.4cm\noindent
$\bullet$ \textbf{Generalized complex structure of type 2.}
Recall that the action of a bivector $\beta$ is given by 
$\rho \mapsto e^{\iota_\beta}\rho$.
If $J$ is a complex structure and $\beta$ is a holomorphic Poisson structure of rank $l$, then one can deform $J$ into a generalized complex structure of type $n-l$ (see \cite{Gualtieri}).
In \cite[Theorem 5.1]{CG-generalized} it is proved that every invariant complex structure on a $2n$-dimensional nilmanifold
can be deformed, via such a $\beta$-field with $l=2$, to get an invariant generalized complex structure of type $n - 2$.
Therefore, our nilmanifolds have a generalized complex structure of type 2.

More concretely, in view of Remark~\ref{hol-Poisson}, from the generalized complex structure of type 4 defined by
$\rho = \theta^1\wedge\theta^2\wedge\theta^3\wedge\theta^4$ above, we get that
$$
\widetilde{\rho} = e^{\theta^3\wedge\theta^4} \, \Omega,
$$
with $\Omega=\theta^1\wedge\theta^2$,
is a generalized complex structure of type 2.
Indeed, $B=\frac{1}{2} (\theta^3\wedge\theta^4 + \overline{\theta^{3}}\wedge \overline{\theta^{4}})$ and
$\omega=-\frac{i}{2} (\theta^3\wedge\theta^4 - \overline{\theta^{3}}\wedge \overline{\theta^{4}})$. Thus,
$$
\omega^2\wedge\Omega\wedge\overline{\Omega} = \frac{1}{2}\,\theta^3\wedge\theta^4\wedge \overline{\theta^{3}}\wedge \overline{\theta^{4}}\wedge\theta^1\wedge\theta^2\wedge \overline{\theta^{1}}\wedge \overline{\theta^{2}} \not= 0,
$$
and $d(\theta^1\wedge\theta^2)=0$ and $d(B+i\, \omega) \wedge \theta^1\wedge\theta^2=
d(\theta^3\wedge\theta^4)\wedge\theta^1\wedge\theta^2=
d(\theta^1\wedge\theta^2\wedge\theta^3\wedge\theta^4)=0$ by the equations~\eqref{FI-equations}.

\vskip.5cm

For the definition of generalized complex structures of type 1 and type 0 we will deal with the space $Z^2(N_{\alpha})$ of invariant closed $2$-forms on the nilmanifold $N_{\alpha}$. The following lemma is straightforward:

\begin{lemma}\label{closed-2-forms}
Every invariant closed $2$-form~$\omega$ on the nilmanifold~$N_{\alpha}$ is given by
\begin{equation}\label{2-forma-cerrada}
\begin{array}{cl}
\omega = \!\!\!&  x_{12}\,e^{12} + x_{13}\,e^{13} + x_{14}\,e^{14} + x_{15}\,e^{15}
	+ x_{16}\,e^{16} + x_{17}\,e^{17} - (1+\alpha)\,x_{37}\,e^{18} \\[5pt]
	&  +\, x_{23}\,e^{23} + x_{24}\,e^{24} + x_{25}\,e^{25} + (1+\alpha)\,x_{17}\,e^{26}
	+ x_{27}\,e^{27} + x_{34}\,e^{34} + x_{17}\,e^{35} + x_{37}\,e^{37} \\[5pt]
	& -\, \big( (1-\alpha)\,x_{16} + (1+\alpha)\,x_{27} \big)\,e^{45}
	+ (1+\alpha)\,(3+\alpha)\,x_{37}\,e^{46} - (1+\alpha)\,x_{37}\,e^{57},
\end{array}
\end{equation}
where $x_{12},\ldots,x_{37} \in \mathbb{R}$. Hence, the space $Z^2(N_{\alpha})$ has dimension $12$.
\end{lemma}

\vskip.4cm\noindent
$\bullet$ \textbf{Generalized complex structure of type 1.}
We must find $\rho = e^{B+i\,\omega} \, \Omega$, with
$B,\omega$ real invariant 2-forms and $\Omega$ a complex $1$-form,
satisfying
$$
\omega^3\wedge\Omega\wedge\overline{\Omega}\not=0,
\quad\quad
d\Omega=0,
\quad\quad
(dB+i\,d\omega)\wedge \Omega=0.
$$

Since $\Omega$ is a complex $1$-form, it can be written as
$\Omega = \sum_{j=1}^4 z_j\, e^j$,
for some complex coefficients $z_1,\ldots,z_4 \in \mathbb{C}$.
Let us choose $\Omega = e^3 + i\, e^4$, which fulfills $d\Omega=0$ according to the structure equations~\eqref{equations}. 

We consider $B=0$ and let $\omega$ be any $2$-form given in~\eqref{2-forma-cerrada} with coefficients $x_{17},\,x_{37}\neq 0$.
A direct calculation shows that
$$
\omega^3\wedge\Omega\wedge\overline{\Omega} 
 = 12\,i\,(1+\alpha)^3\,x_{17}\,x_{37}^2\,e^{12345678} \neq 0.
$$
Since both $B$ and $\omega$ are closed, the condition $(dB+i\,d\omega)\wedge \Omega=0$ is 
trivially fulfilled, and $\rho=e^{i\,\omega} \, \Omega$ defines a generalized complex structure
of type $1$ on the nilmanifold $N_{\alpha}$.

\vskip.4cm\noindent
$\bullet$ \textbf{Generalized complex structure of type 0 (symplectic structure).}
The form $\omega$ in Lemma~\ref{closed-2-forms} determined by \eqref{2-forma-cerrada} satisfies
\begin{equation}\label{no-deg}
\begin{split}
\omega^4 = - 24\,(1+\alpha)^2\,x_{37}^2\,\bigg(
	&\Big( (1-\alpha)\,x_{16} - 2\,x_{27} + (1+\alpha)\,x_{34} \Big)x_{17}\\
	&+\Big( (1+\alpha)\,(3+\alpha)\,x_{23} + (3+\alpha)\,x_{25} \Big)x_{37}\,
	\bigg)\,e^{12345678}.
\end{split}
\end{equation}
It suffices to choose, for instance, $x_{17}=x_{23}=0$ and $x_{25}x_{37}\neq 0$ to get
a symplectic form $\omega$ on $N_{\alpha}$.

\medskip
This concludes the proof of Proposition~\ref{prop-1}.

\section{The nilmanifolds $N_{\alpha}$ and their minimal model}\label{no-iso}

\noindent
The goal of this section is to prove Proposition~\ref{prop-2}, i.e. the nilmanifolds $N_{\alpha}$ and $N_{\alpha'}$ have non-isomorphic $\mathbb{R}$-minimal models for $\alpha \not= \alpha'$.

As we recalled in Section~\ref{nilvariedades}, 
the $\mathbb{R}$-minimal model of the nilmanifold $N_{\alpha}$
is given by the Chevalley-Eilenberg complex $(\bigwedge\frg_{\alpha}^*,d)$ of its underlying Lie algebra $\frg_{\alpha}$. Consequently, to prove Proposition~\ref{prop-2} it suffices to show
that the real Lie algebras $\frg_{\alpha}$, $\alpha \in {\mathbb{Q}}^+$, define a family of pairwise non-isomorphic
nilpotent Lie algebras. Indeed, we will prove the following

\begin{proposition}\label{prop-no-iso-NLA}
If the nilpotent Lie algebras $\frg_{\alpha}$ and $\frg_{\alpha'}$ are isomorphic, 
then $\alpha = \alpha'$.
\end{proposition}

Remember that in eight dimensions, no classification of nilpotent Lie algebras is available. 
Indeed, nilpotent Lie algebras are classified only up to real dimension 7. 
More concretely, Gong classified in~\cite{Gong} the 7-dimensional nilpotent Lie algebras in 140 algebras 
together with 9 one-parameter families. One can check that our family $\{\frg_{\alpha}\}_{\alpha\in {\mathbb{Q}}^+}$ is not an extension of any of those 
9 families, i.e. the quotient of $\frg_{\alpha}$ by its center (which has dimension~$1$) does not belong to any of the 9 one-parameter families of Gong. 
Furthermore, the usual invariants for nilpotent Lie
algebras are the same for all the algebras in the family $\{\frg_{\alpha}\}_{\alpha\in {\mathbb{Q}}^+}$. 
For instance,
the dimensions of the terms in the ascending central series 
are~$(1,3,6,8)$,
whereas those of the descending central series are~$(4,3,1,0)$ (see Lemma~\ref{lema-F} for further details). 
Moreover, the dimensions of the Lie algebra cohomology groups 
$H^k(\frg_{\alpha})$ coincide for every $\alpha\in {\mathbb{Q}}^+$, as shown at the end of Section~\ref{nilvariedades}. 
For this reason, we will directly analyze the existence of an isomorphism between
any two of the nilpotent Lie algebras in our family $\{\frg_{\alpha}\}_{\alpha\in {\mathbb{Q}}^+}$. 

\vskip.2cm

The following technical lemma will be useful for our purpose.

\begin{lemma}\label{lema-general}
Let $f\colon\frg\longrightarrow\frg'$ be an isomorphism of the Lie algebras $\frg$ and $\frg'$.
Consider an ideal $\{0\}\not= \fra \subset \frg$, and let $\fra'=f(\fra) \subset \frg'$
be the corresponding ideal in $\frg'$. Let $\{e_{r+1},\ldots,e_m\}$, resp. $\{e'_{r+1},\ldots,e'_m\}$, be a basis of $\fra$,
resp. $\fra'$,
and complete it up to a basis $\{e_1,\ldots,e_r,e_{r+1},\ldots,e_m\}$ of $\frg$, resp. $\{e'_1,\ldots,e'_r,e'_{r+1},\ldots,e'_m\}$ of $\frg'$.
Denote the dual bases of $\frg^*$ and $\frg'^*$ respectively by $\{e^i\}_{i=1}^m$
and $\{e'^{\,i}\}_{i=1}^m$. Then, the dual map $f^*\colon\frg'^*\longrightarrow\frg^*$
satisfies
$$
f^*(e'^{\,i}) \wedge e^{1} \wedge \ldots \wedge e^{r}=0, \qquad \mbox{ for all }\ i=1,\ldots,r.
$$
\end{lemma}

\begin{proof}
Let $\pi\colon\frg\longrightarrow\frg/\fra$ and $\pi'\colon\frg'\longrightarrow\frg'/\fra'$ be the
natural projections, and $\tilde{f}\colon\frg/\fra\longrightarrow\frg'/\fra'$ the Lie algebra isomorphism induced by $f$
on the quotients. Taking the corresponding dual maps, we have the following commutative diagrams:
\[
\xymatrix{
\frg \ar[d]_{\pi} \ar[r]^{f} & \frg' \ar[d]^{\pi'} \\
\frg/\fra \ar[r]^{\tilde{f}} & \frg'/\fra',
}
\qquad
\qquad\qquad
\xymatrix{
\frg'^{\,*}  \ar[r]^{f^*} & \frg^*  \\
(\frg'/\fra')^* \ar[u]^{\pi'^{\,*}} \ar[r]^{\tilde{f}^*} & (\frg/\fra)^* \ar[u]_{\pi^*}.
}
\]
Taking the basis $\{e_1,\ldots,e_r,e_{r+1},\ldots,e_m\}$ of $\frg$, we have that $\{\tilde{e}_1,\ldots,\tilde{e}_r\}$
is a basis of $\frg/\fra$. Let
$\{\tilde{e}^1,\ldots,\tilde{e}^r\}$ be its dual basis for $(\frg/\fra)^*$.
Using a similar procedure, we find a basis $\{\tilde{e}'^{\,1},\ldots,\tilde{e}'^{\,r}\}$ for $(\frg'/\fra')^*$.
Since the maps $\pi^*$ and $\pi'^{\,*}$ are injective, and the diagram is commutative, we get
$$
f^*(e'^{\,i})=f^*\big(\pi'^{\,*}(\tilde{e}'^{\,i})\big)=\pi^*\big(\tilde{f}^*(\tilde{e}'^{\,i})\big) \in \langle e^1,\ldots, e^r\rangle,
$$
for any $1\leq i \leq r$.
\end{proof}

Applying the previous result to our particular case, we get:

\begin{lemma}\label{lema-F}
Consider $\frg_{\alpha}$ and $\frg_{\alpha'}$ for $\alpha,\alpha' \in \mathbb{Q}^+$.
If $f\colon\frg_{\alpha}\longrightarrow\frg_{\alpha'}$ is an isomorphism of Lie algebras, then in terms of their
respective bases $\{e^i\}_{i=1}^8$ and $\{e'^{\,i}\}_{i=1}^8$ given in~\eqref{equations},
the dual map $f^*\colon\frg_{\alpha'}^*\longrightarrow\frg_{\alpha}^*$ satisfies
\begin{equation}\label{cambio-general-reducido}
\begin{split}
&f^*(e'^{\,i}) \wedge e^{12} =0, \text{ \ for } i=1,2,\\
&f^*(e'^{\,i}) \wedge e^{1234} =0, \text{ \ for } i=3,4,
\end{split}
\qquad\quad
\begin{split}
&f^*(e'^{\,5}) \wedge e^{12345} =0, \\
&f^*(e'^{\,i}) \wedge e^{1234567} =0,  \text{ for } i=6,7.
\end{split}
\end{equation}
\end{lemma}

\begin{proof}
Recall that the ascending central series of a Lie algebra $\frg$ is defined by $\{\frg_k\}_{k}$, where
$\frg_0=\{0\}$ and
$$
\frg_k=\{X\in\frg \mid [X,\frg]\subseteq \frg_{k-1}\}, \text{ for } k\geq 1.
$$
Observe that $\frg_1=Z(\frg)$ is the center of~$\frg$.

Let $\{e_i\}_{i=1}^8$ and $\{e'_i\}_{i=1}^8$ be the bases for $\frg_{\alpha}$ and $\frg_{\alpha'}$ dual to
$\{e^i\}_{i=1}^8$ and $\{e'^{\,i}\}_{i=1}^8$, respectively.
In terms of these bases, the ascending central series of $\frg_{\alpha}$ and $\frg_{\alpha'}$ are
$$
(\frg_{\alpha})_1=\langle e_8 \rangle \subset
(\frg_{\alpha})_2=\langle e_6,e_7,e_8 \rangle \subset
(\frg_{\alpha})_3=\langle e_3,e_4,e_5,e_6,e_7,e_8 \rangle,
$$
and
$$
(\frg_{\alpha'})_1=\langle e'_8 \rangle \subset
(\frg_{\alpha'})_2=\langle e'_6,e'_7,e'_8 \rangle \subset
(\frg_{\alpha'})_3=\langle e'_3,e'_4,e'_5,e'_6,e'_7,e'_8 \rangle.
$$
Since $f\big((\frg_{\alpha})_k\big)=(\frg_{\alpha'})_k$ for any Lie algebra isomorphism $f\colon\frg_{\alpha}\longrightarrow\frg_{\alpha'}$,
applying Lemma~\ref{lema-general} to the ideals $\mathfrak{a}=(\frg_{\alpha})_k$ for $k=1,2,3$
one gets~\eqref{cambio-general-reducido} for $i=1,2,5,6$ and $7$.

Moreover, the derived algebras of $\frg_{\alpha}$ and $\frg_{\alpha'}$ are, respectively,
$$
[\frg_{\alpha},\frg_{\alpha}]=\langle e_5,e_6,e_7,e_8 \rangle,
\qquad
[\frg_{\alpha'},\frg_{\alpha'}]=\langle e'_5,e'_6,e'_7,e'_8 \rangle.
$$
Using again Lemma~\ref{lema-general} with $\mathfrak{a}=[\frg_{\alpha},\frg_{\alpha}]$,
we obtain~\eqref{cambio-general-reducido} for $i=3,4$.
\end{proof}

\vskip.4cm

We are now in the conditions to prove Proposition~\ref{prop-no-iso-NLA}.

\bigskip
\noindent{\textbf{Proof of Proposition~\ref{prop-no-iso-NLA}.}}
Given any homomorphism of Lie algebras $f \colon \frg_{\alpha} \longrightarrow \frg_{\alpha'}$,
its dual map $f^* \colon \frg_{\alpha'}^* \longrightarrow \frg_{\alpha}^*$ naturally extends to a map
$F \colon \bigwedge^*\frg_{\alpha'}^* \longrightarrow \bigwedge^*\frg_{\alpha}^*$ that commutes with the differentials,
i.e. $F\circ d=d\circ F$.
Hence, in terms of the bases $\{e^i\}_{i=1}^8$ for~$\frg_{\alpha}^*$ and $\{e'^{\,i}\}_{i=1}^8$ 
for $\frg_{\alpha'}^*$ 
satisfying the equations~\eqref{equations}
with respective parameters~$\alpha$ and~$\alpha'$, any Lie algebra isomorphism is defined by
\begin{equation}\label{cambio-base}
F(e'^{\,i}) =\sum_{j=1}^8 \lambda_j^i\, e^j, \quad i=1,\ldots,8,
\end{equation}
satisfying conditions
\begin{equation}\label{condiciones}
d\big(F(e'^{\,i})\big)-F(de'^{\,i})  =0, \ \text{ for each }1\leq i\leq 8,
\end{equation}
where the matrix $\Lambda=(\lambda^i_j)_{i,j=1,\ldots,8}$ belongs to ${\rm GL}(8,\mathbb{R})$.

We first note that the preceding lemma allows us to simplify the $8\times 8$ matrix $\Lambda$.
In fact, from Lemma~\ref{lema-F} one has that
$\lambda^i_j =0$ for $1\leq i \leq 2$ and $3\leq j \leq 8$,
for $3\leq i \leq 5$ and $6\leq j \leq 8$, and also
$\lambda^3_5 = \lambda^4_5 = \lambda^6_8 =\lambda^7_8 =0$.
Since $\Lambda$ belongs to ${\rm GL}(8,\mathbb{R})$, the previous conditions imply that 
$\lambda^5_5\neq 0$ and $\lambda^8_8\neq 0$.

Note also that~\eqref{condiciones} is   
trivially fulfilled for $1\leq i\leq 4$. 
Hence, it suffices to focus on $5\leq i\leq 8$.
We will denote by $\big[d\big(F(e'^{\,i})\big)-F(de'^{\,i})\big]_{jr}$ the coefficient
for $e^{jr}$ in the expression of the 2-form $d\big(F(e'^{\,i})\big)-F(de'^{\,i})$.

By a direct calculation we have 
$$
0=\big[d\big(F(e'^{\,8})\big)-F(de'^{\,8})\big]_{35}=2\,\lambda^4_3\,\lambda^5_5.
$$
Since $\lambda^5_5\neq 0$, we conclude that $\lambda^4_3=0$.
Now observe that the following expressions must annihilate:
\begin{equation*}
\begin{aligned}
\big[d\big(F(e'^{\,6})\big)-F(de'^{\,6})\big]_{23} &=
	- (\lambda^1_2 \lambda^5_3 + \lambda^6_7),\\
\big[d\big(F(e'^{\,7})\big)-F(de'^{\,7})\big]_{23} &=
	\ \lambda^2_2\big(\lambda^3_3  -(1+\alpha')  \lambda^5_3\big) - \lambda^7_7,\\
\big[d\big(F(e'^{\,6})\big)-F(de'^{\,6})\big]_{25} &=
	-\lambda^1_2 \lambda^5_5 + (1+\alpha) \lambda^6_7, \\
\big[d\big(F(e'^{\,7})\big)-F(de'^{\,7})\big]_{25} &=
	-(1+\alpha')  \lambda^2_2\lambda^5_5   +(1+\alpha)  \lambda^7_7.
\end{aligned}
\end{equation*}
Solving $\lambda^6_7$ and $\lambda^7_7$ from the first two equations and
replacing their values in the last ones, we get:
\begin{equation}\label{ecuaciones1}
\lambda^1_2\big(\lambda^5_5 +(1+\alpha) \,\lambda^5_3\big) = 0,
\qquad
\lambda^2_2\left(\lambda^3_3 -(1+\alpha')\,\lambda^5_3 - \frac{1+\alpha'}{1+\alpha}\,\lambda^5_5\right) = 0.
\end{equation}
Moreover, also the following terms must vanish:
$$
\big[d\big(F(e'^{\,8})\big)-F(de'^{\,8})\big]_{34}
 = -\lambda^4_4\,(\lambda^3_3  +2  \lambda^5_3) + \lambda^8_8,
\qquad
\big[d\big(F(e'^{\,8})\big)-F(de'^{\,8})\big]_{45}
 = 2\,(\lambda^4_4 \lambda^5_5 -\lambda^8_8).
$$
From the second one, we have $\lambda^8_8=\lambda^4_4 \lambda^5_5$. 
Since $\lambda^8_8 \neq 0$, in particular also $\lambda^4_4\neq 0$.
Using the first expression above, we can then solve
\begin{equation}\label{ecuaciones2}
\lambda^3_3=\lambda^5_5 - 2\,\lambda^5_3.
\end{equation}
In addition, observe that
$\big[d\big(F(e'^{\,5})\big)-F(de^{5})\big]_{12}=\big[d\big(F(e'^{\,7})\big)-F(de^{7})\big]_{13}=0$ lead to
\begin{equation}\label{ecuaciones3}
\lambda^5_5=\lambda^1_1\lambda^2_2 - \lambda^1_2\lambda^2_1,
\qquad
\lambda^2_1\big(\lambda^3_3 -(1+\alpha')  \lambda^5_3 \big)=0.
\end{equation}

We now check that the vanishing of the coefficient $\lambda^1_1$ leads to a contradiction.
Indeed, in such case, the first expression in \eqref{ecuaciones3}
becomes $\lambda^5_5=-\lambda^1_2\lambda^2_1\neq 0$, and
from \eqref{ecuaciones1} we then have $\lambda^5_5 =-(1+\alpha) \,\lambda^5_3$,
which plugged into \eqref{ecuaciones2} gives $\lambda^3_3=-(3+\alpha)\lambda^5_3$.
Replacing this value in the second equation of \eqref{ecuaciones3}, the condition
$\lambda^5_3(\alpha +\alpha' + 4) = 0$ arises.
Since $\alpha$ and $\alpha'$ are greater than zero, we are forced to consider $\lambda^5_3 = 0$.
However, this leads to $\lambda^5_5 = 0$,
which is a contradiction.
Hence, we necessarily have that $\lambda^1_1$ is nonzero.

Since $\lambda^1_1\neq 0$, the condition
$$
0=\big[d\big(F(e'^{\,6})\big)-F(de'^{\,6})\big]_{13} = - \lambda^1_1 \lambda^5_3
$$
implies $\lambda^5_3 = 0$. Replacing this value in \eqref{ecuaciones1}, \eqref{ecuaciones2}, and
\eqref{ecuaciones3}, we obtain:
$$
\lambda^1_2\,\lambda^5_5=0, \qquad
\lambda^2_1\,\lambda^5_5=0, \qquad
\lambda^2_2\,\lambda^5_5\left( 1- \frac{1+\alpha'}{1+\alpha}\right)=0,\qquad
\lambda^3_3=\lambda^5_5=\lambda^1_1\lambda^2_2 - \lambda^1_2\lambda^2_1.
$$
As $\lambda^5_5\neq 0$, one immediately has $\lambda^2_1 = \lambda^1_2 = 0$ and
$\lambda^5_5=\lambda^1_1\lambda^2_2$. Consequently $\lambda^2_2\neq 0$, which
allows us to conclude $1+\alpha  = 1+\alpha'$, and
thus $\alpha = \alpha'$.
This completes the proof of the proposition.
\hfill$\square$

\begin{remark}\label{C-minimal-model}
{\rm
In addition to the notions of rational and real homotopy types, there is also the notion of \emph{complex} homotopy type \cite{DGMS}. Two manifolds $X$ and $Y$ have the same $\mathbb{C}$-homotopy type if and only if their $\mathbb{C}$-minimal models $(\bigwedge V_X \otimes_{\mathbb{Q}} \mathbb{C}, d)$ and $(\bigwedge V_Y \otimes_{\mathbb{Q}} \mathbb{C}, d)$  are isomorphic. Here $(\bigwedge V_X, d)$ and $(\bigwedge V_Y, d)$ are the rational minimal models of $X$ and $Y$, respectively. Recall that when the field $\mathbb{K}$ has ${\rm char}\, (\mathbb{K})=0$, the $\mathbb{K}$-minimal model is unique up to isomorphism. 
Clearly, if $X$ and $Y$ have different complex homotopy type, then $X$ and $Y$ have different real (hence, also rational) homotopy type.

For nilmanifolds, it is proved in \cite[Theorem 2]{BM2012}
that there are exactly $30$ complex homotopy types of $6$-dimensional nilmanifolds. It is worth remarking that if $\alpha \not= \alpha'$, then our nilmanifolds $N_{\alpha}$ and~$N_{\alpha'}$ have different $\mathbb{C}$-minimal model. Indeed, it can be checked that the proof of Proposition~\ref{prop-2} above directly
extends
to the case when the matrix $\Lambda=(\lambda^i_j)_{i,j=1,\ldots,8}$ defined in~(\ref{cambio-base}) belongs to ${\rm GL}(8,\mathbb{C})$.
In conclusion, our main result in Theorem~\ref{infinite-GCS} extends to the complex case, i.e. 
there are infinitely many \emph{complex} homotopy types of $8$-dimensional nilmanifolds admitting
a generalized complex structure of every type $k$, for $0 \leq k \leq 4$. 
Now, Corollary~\ref{cor-infinite-GCS} is a consequence of the fact that the product
nilmanifolds $N_{\alpha}\times {\mathbb T}^{2m}$ do not admit any K\"ahler metric. 
}
\end{remark}

\medskip

\noindent{\textbf{Acknowledgments}}.
This work has been partially supported by the projects MTM2017-85649-P (AEI/FEDER, UE), and
E22-17R ``\'Algebra y Geometr\'{\i}a'' (Gobierno de Arag\'on/FEDER).


\smallskip

\begin{thebibliography}{33}

\bibitem{ACK}
D. Angella, S. Calamai, H. Kasuya,
Cohomologies of generalized complex manifolds and nilmanifolds,
\emph{J. Geom. Anal.} \textbf{27} (2017), 
142--161.

\bibitem{BFLM}
G. Bazzoni, M. Freibert, A. Latorre, B. Meinke,
Complex symplectic structures on Lie algebras,
arXiv:1811.05969 [math.SG],

\bibitem{BM2012}
G. Bazzoni, V. Mu\~noz,
Classification of minimal algebras over any field up to dimension 6,
\emph{Trans. Amer. Math. Soc.} \textbf{364} (2012), 
1007--1028.


\bibitem{CG-generalized}
G. Cavalcanti, M. Gualtieri,
Generalized complex structures on nilmanifolds,
\emph{J. Symplectic Geom.} \textbf{2} (2004), 
393--410.


\bibitem{DGMS}
P. Deligne, P. Griffiths, J. Morgan, D. Sullivan,
Real homotopy theory of K\"ahler manifolds,
\emph{Invent. Math.} \textbf{29} (1975), 245--274.

\bibitem{FHT-libro}
Y. F\'elix, S. Halperin, J.C. Thomas,
\emph{Rational Homotopy Theory},
Graduate Texts in Mathematics 205, Springer, 2001.

\bibitem{Gong} 
M-P. Gong, 
Classification of nilpotent Lie algebras of dimension $7$ (over algebraically
closed fields and $\mathbb{R}$), 
Ph. D. Thesis, University of Waterloo, Ontario, Canada, 1998.

\bibitem{GM-libro}
P. Griffiths, J. Morgan,
\emph{Rational Homotopy Theory and Differential Forms},
Progress in Mathematics, Birkh\"auser, 1981.

\bibitem{Gualtieri}
M. Gualtieri, 
Generalized complex geometry, 
\emph{Ann. of Math. (2)} \textbf{174} (2011), 75--123.

\bibitem{Has-PAMS}
K. Hasegawa,
Minimal models of nilmanifolds,
\emph{Proc. Amer. Math. Soc.} \textbf{106} (1989), 65--71.

\bibitem{Hitchin}
N. J. Hitchin, 
Generalized Calabi-Yau manifolds, 
\emph{Q. J. Math.} \textbf{54} (2003), 281--308.

\bibitem{LUV-SnN}
A. Latorre, L. Ugarte, R. Villacampa,
The ascending central series of nilpotent Lie algebras with complex structure,
\emph{Trans. Amer. Math. Soc.}  \textbf{372} (2019), 3867--3903.

\bibitem{LUV-ComplexManifolds}
A. Latorre, L. Ugarte, R. Villacampa,
A family of complex nilmanifolds with infinitely many real homotopy types,
\emph{Complex Manifolds} \textbf{5} (2018), 89--102.

\bibitem{Malcev} I.A. Mal'cev,
A class of homogeneous spaces,
Amer. Math. Soc. Transl. 39 (1951).

\bibitem{Nomizu} K. Nomizu,
On the cohomology of compact homogeneous spaces of nilpotent Lie groups,
\emph{Ann. Math.} \textbf{59} (1954), 531--538.

\bibitem{Salamon}
S. Salamon,
Complex structures on nilpotent Lie algebras,
\emph{J. Pure Appl. Algebra} \textbf{157} (2001), 311--333.

\bibitem{Sullivan}
D. Sullivan,
\emph{Infinitesimal Computations in Topology},
Inst. Hautes \'Etudes Sci. Publ. Math.  47 (1977), 269--331.

\bibitem{Thu} W.P. Thurston, 
Some simple examples of symplectic manifolds,
\emph{Proc. Amer. Math. Soc.} {\bf 55} (1976), 467--468.

\bibitem{TO-libro} A. Tralle, J. Oprea, 
\emph{Symplectic manifolds with no K\"ahler structure}, 
Lecture Notes in Math. 1661, Springer 1997.

\end{thebibliography}
\end{document}